\documentclass[12pt,a4paper,leqno,verbatim]{amsart}
\newcounter{minutes}
\divide\time by 60
\newcounter{hours}
\multiply\time by 60 \addtocounter{minutes}{-\time}

\usepackage{amssymb}
\usepackage{hyperref}
\usepackage{graphicx}
\usepackage{color}

\date{}
\newfont{\cyrilic}{wncyr10 scaled 1000}
\title[On two new means of two arguments III]
{On two new means of two arguments III}

\author[B.A Bhayo]{Barkat Ali Bhayo}
\address{Department of Mathematics, Sukkur IBA University, Airport road Sukkur, Pakistan}
\email{bhayo.barkat@gmail.com}

\author[J. S\'andor]{J\'ozsef S\'andor}
\address{Babe\c{s}-Bolyai University
Department of Mathematics
Str. Kogalniceanu nr. 1
400084 Cluj-Napoca, Romania}
\email{jsandor@math.ubbcluj.ro}

\newcommand{\comment}[1]{}

\swapnumbers
\theoremstyle{plain}

\newtheorem{theorem}[equation]{Theorem}
\newtheorem{theorem a}[equation]{Theorem A}
\newtheorem{theorem b}[equation]{Theorem B}
\newtheorem{lemma}[equation]{Lemma}

\newtheorem{corollary}[equation]{Corollary}

\newtheorem{remark}[equation]{Remark}

\numberwithin{equation}{section}

\begin{document}


%
%
\def\thefootnote{}
\footnotetext{ \texttt{\tiny File:~\jobname .tex,
          printed: \number\year-\number\month-\number\day,
          \thehours.\ifnum\theminutes<10{0}\fi\theminutes}
} \makeatletter\def\thefootnote{\@arabic\c@footnote}\makeatother


\begin{abstract}
In this paper authors establish the two sided inequalities for the 
following two new means
$$X=X(a,b)=Ae^{G/P-1},\quad  Y=Y(a,b)=Ge^{L/A-1}.$$
As well as many other well known inequalities involving the identric mean $I$ and the logarithmic mean are refined from the literature as an application.
\end{abstract}

\maketitle

\bigskip
{\bf 2010 Mathematics Subject Classification}: 26D05, 26D15, 26D99.

{\bf Keywords}: Inequalities, means of two arguments, identric mean, logarithmic mean, $X$ and $Y$ means.


\section{\bf Introduction}
The study of the inequalities involving the classical means
such as arithmetic mean $A$, geometric mean $G$, identric mean $I$ and logarithmic mean $L$  have been of the extensive interest for several authors, e.g., see \cite{alzer1, alzer2,carlson,hasto,ns1004, ns1004a,sandorc,sandord,sandore,vvu}.

In 2011, S\'andor \cite{sandora}  introduced a new mean
$X(a,b)$ for two positive real numbers $a$ and $b$,
defined by
$$X=X(a,b)=Ae^{G/P-1},$$
where
$A=A(a,b)=(a+b)/2,\, G=G(a,b)=\sqrt{ab},$ and
$$P=P(a,b)=\frac{a-b}{2\displaystyle\arcsin\left(\frac{a-b}{a+b}\right)},\quad a\neq b,$$
are the arithmetic mean, geometric mean, and Seiffert mean 
\cite{seiff}, respectively. This paper contains essentially results on the $X$ mean, where several inequalities involving the $X$ mean and the refinement of \eqref{89ineq}and \eqref{89inequ} are established.

In the same paper, S\'andor introduced an other mean $Y(a,b)$ 
for two positive real $a$ and $b$, which is defined by
$$Y=Y(a,b)=Ge^{L/A-1},$$
where
$$L=L(a,b)=\frac{a-b}{\log(a)-\log(y)},\quad a\neq b,$$
is a logarithmic mean. For two positive real numbers $a$ and $b$, the identric mean and harmonic mean are defined by
$$I=I(a,b)=\frac{1}{e}\left(\frac{a^a}{b^b}\right)^{1/(a-b)},\quad a\neq b,$$
and 
$$H=H(a,b)=2ab/(a+b),$$
respectively.
In 2012, the $X$ mean appeared in \cite{sandortwosharp}. In 2014, $X$ and $Y$ means published in the journal of Notes in Number Theory and. Discrete Mathematics \cite{sandornew}.
For the historical background and the generalization of these means we refer the reader to see, e.g, \cite{alzer2,carlson,mit,ns1004,ns1004a,sandor611,sandorc,sandord,sandore,sandor1405,sandorF,vvu}. 
Connections of these means with the trigonometric or hyperbolic inequalities can be found in \cite{barsan,sandora,sandornew,sandore}.

In \cite{sandornew}, S\'andor  proved inequalities of $X$ and $Y$ means in terms of other classical means as well as their relations with each other. Some of the inequalities are recalled for the easy reference. 

\begin{theorem}\label{sandortheorem}\cite{sandornew} For $a,b>0$ with $a\neq b$, one has
\begin{enumerate}
\item $\displaystyle G<\frac{AG}{P}<X<\frac{AP}{2P-G}<P$,\\
\item $\displaystyle H<\frac{LG}{A}<Y<\frac{AG}{2A-L}<G$,\\
\item $\displaystyle 1<\frac{L^2}{IG}<\frac{L\cdot e^{G/L-1}}{G}<\frac{PX}{AG}$,\\
\item $\displaystyle H<\frac{G^2}{I}<\frac{LG}{A}<\frac{G(A+L)}{3A-L}<Y$.\\
\end{enumerate}
\end{theorem}

In \cite{barsan}, a series expansion of $X$ and $Y$ was given and proved the following inequalities.
\begin{theorem}\label{sandorbhayo-theorem}\cite{barsan} For $a,b>0$ with $a\neq b$, one has
\begin{enumerate}
\item $\displaystyle \frac{1}{e}(G+H)< Y<\frac{1}{2}(G+H),$\\
\item $G^2I<IY<IG<L^2,$\\
\item $\displaystyle\frac{G-Y}{A-L} < \frac{Y+G}{2A}<  \frac{3G+H}{4A}  <1,$\\
\item $\displaystyle L<\frac{2G+A}{3}<X<L(X,A)<P<\frac{2A+G}{3}<I,$\\
\item $\displaystyle 2\left(1- \frac{A}{P}\right)<\log \left(\frac{X}{A}\right) <
 \left(\frac{P}{A}\right)^2.$
\end{enumerate}
\end{theorem}

For $p\in \mathbb{R}$ and $a,b>$ with $a\neq b$, the $p$th power mean 
$M_p(a,b)$ and $p$th power-type Heronian mean $H_p$(a,b) are define by 

$$
M_p=M_p(a,b)=
\begin{cases}
\displaystyle\left(\frac{a^p+b^p}{2}\right)^{1/p}, & p\neq 0,\\
\sqrt{ab}, & p=0,
\end{cases}
$$
and 
$$
H_p=H_p(a,b)=
\begin{cases}
\displaystyle\left(\frac{a^p+(ab)^{p/2}+b^p}{3}\right)^{1/p}, & p\neq 0,\\
\sqrt{ab}, & p=0,
\end{cases}
$$
respectively.

In \cite{chuet}, Chu et al. proved that the following double inequality 
\begin{equation}\label{89ineq}
M_p<X<M_q
\end{equation}
holds for all $a,b>0$ with $a\neq b$ if and only if $p\leq 1/3$ and
$q\geq \log (2)/(1+\log(2))\approx 0.4093.$

Recently, Zhou et al. \cite{zhouet} proved that for all $a,b>0$ with $a\neq b$,
the following double inequality
\begin{equation}\label{89inequ}
H_\alpha < X < H_\beta
\end{equation}
holds if and only if $\alpha \leq 1/2$ and $\beta\geq \log(3)/(1+\log(2))
\approx 0.6488$.

This paper is organized as follows: In Section 1, we give the introduction. Section 2 consists of main results and remarks. In Section 3, some connections of $X$, $Y$ and other means are given with trigonometric and hyperbolic functions. Some lemmas are also proved in this section which will be used in the proof of main result. Section 4 deals with the proof of the main result and corollaries.

\section{\bf Main result and motivation}

Making contribution to the topic, authors refine some previous results appeared in the literature \cite{alzer0, alzer1, barsan, chuet, zhouet,sandornew} as well as establish new results involving the $X$ mean. 

\begin{theorem}\label{thm1} For $a,b>0$, we have
\begin{equation}\label{30ineqa}
\alpha G+(1-\alpha)A<X<\beta G+(1-\beta)A,
\end{equation}
with best possible constants $\alpha=2/3\approx 0.6667$ and $\beta =(e-1)/e\approx 0.6321$,
and
\begin{equation}\label{30ineqb}
A+G-\alpha_1 P<X<A+G-\beta_1 P,
\end{equation}
with best possible constants $\alpha_1=1$ and $\beta_1 =\pi(e-1)/(2e)\approx 0.9929$.
\end{theorem}

\begin{remark} \rm In \cite[Theorem 2.7]{sandornew}, S\'andor proved that for $a\neq b$,
\begin{equation}\label{ine1502a}
X<A\left[\frac{1}{e}+\left(1-\frac{1}{e}\right)\frac{G}{P}\right],
\end{equation}
and
\begin{equation}\label{ine1502b}
Y<G\left[\frac{1}{e}+\left(1-\frac{1}{e}\right)\frac{L}{A}\right].
\end{equation}
As $A/P >1$, the right side of \eqref{30ineqa} gives a slight improvement to \eqref{ine1502a}.
From \eqref{ine1502b}, as clearly  $G\cdot L/A <A$, we get a similar inequality. The second inequality in \eqref{30ineqb} could be a counterpart of the inequality $L+G-A<Y$ studied in 
\cite[Theorem 20]{barsan}.
\end{remark}



H. Alzer \cite{alzer0} proved the following inequalities:
\begin{equation}\label{2602alzer}
1<(A+G)/(L+I)<  e/2,
\end{equation}
where the constants $1$ and $2/e$ are best possible. The following result improves among others the right side of \eqref{2602alzer}.

\begin{theorem}\label{2602thm0} For $a\neq b$, one has
\begin{equation}\label{2402a}
(A+G)/e< X< M_q < (L+I)/2 < (A+G)/2,
\end{equation}
 where $q= \log (2)/(1+\log (2))\approx 0.4094$
is the best possible constant.
\end{theorem}

\begin{remark} \rm Particularly, \eqref{2402a} implies that
\begin{equation}\label{2402b}
X< (L+I)/2,  
\end{equation}
which is new.
Since  $L<X< I$  (see Theorem \ref{sandortheorem} and \ref{sandorbhayo-theorem}),  $X$  is below the arithmetic mean of $L$ and $I$. In fact, by left side of \eqref{89ineq}, and by $L<M_{1/3}$ and 
$L<I < M_{2/3}$, we get also
\begin{equation}\label{2402c}
L< M_{1/3}< X<M_q< (L+I)/2 <I < M_{2/3}
\end{equation}
\end{remark}

\begin{theorem}\label{2602thm1} For $a\neq b$, one has
\begin{equation}\label{2402ee}
A+G-P< X< P^2/A < (A+G)/2.
\end{equation} 
\end{theorem}

\begin{remark}\rm
The right side of \eqref{2402ee} offers another refinement  to  $X< (A+G)/2$. An improvement of
 $P^2>XA$ appears in \cite[Theorem 2.9]{sandornew}:
$$P^2> (A^2((A+G)/2)^4)^{1/3} > AX,$$
so \eqref{2402ee} could be further refined. For the following inequalities
\begin{equation}\label{2402e}
L< \frac{2G+A}{3}<A+G-P<X<\sqrt{PX}<\frac{A+G}{2}
\end{equation} 
$$< \frac{P+X}{2}<P< \frac{2A+G}{3}<I,$$
one can see that the first inequality is Carlson's inequality, while the second written in the form 
$P<(2A+G)/3$ is due to S\'andor \cite{sandor1405}. The third inequality is Theorem 2.10 in \cite{sandornew}, while the fourth, written as $PX< ((A+G)/2)^2$ is Theorem 2.11 of \cite{sandornew}. The inequality $(P+X)/2<P$ follows by $X<P$, while the last two inequalities are due to S\'andor 
(\cite{sandor1405, sandord}).
\end{remark}


\begin{theorem}\label{2602thm2} For $a\neq b$, one has
\begin{equation}\label{2402d}
M_{1/2}< (P+X)/2< M_k,     
\end{equation}
where  $k= (5\log 2+2)/(6(\log 2 +1))\approx 0.5380$.
\end{theorem}

\begin{remark}\rm
One has
\begin{equation}\label{2402f}
L< \frac{2G+A}{3}<X<\frac{L+I}{2}<\frac{A+G}{2}< \frac{P+X}{2}<P< \frac{2A+G}{3}<I.
\end{equation} 
and 
\begin{equation}\label{2402g}
\sqrt{AG}<\sqrt{PX}< \frac{A+G}{2}.
\end{equation} 
Relation \eqref{2402g} shows  that  $\sqrt{PX}$ lies between the geometric and arithmetic means of $A$ and $G$;
while \eqref{2402e} shows among others that $(A+G)/2$ lies between the geometric and arithmetic means of $P$ and $X$.
\end{remark}



\begin{theorem}\label{2602-thm} One has
\begin{equation}\label{2602-9}
M_p\leq M_{1/3}< (2G+A)/3 < X,\quad {\rm for}\quad p\leq 1/3,
\end{equation}
\begin{equation}\label{2602-10}
H_\alpha\leq  H_{1/2}< (2G+A)/3 <X,  \quad {\rm for}\quad \alpha \leq 1/2.
\end{equation}
\end{theorem}



\begin{theorem}\label{thm30} 
For $a\neq b$, one has
$$(AX)^{1/\alpha_2}<P<(AX^{\beta_2})^{1/(1+\beta_2)},$$
with best possible constants $\alpha_2=2$ and $\beta_2=\log(\pi/2)/\log(2e/\pi)\approx 0.8234.$
\end{theorem}



\section{\bf Preliminaries and lemmas}

The following result by Biernacki and Krzy\.z \cite{bier} will
be used in studying the monotonicity of certain power series.

\begin{lemma}\label{lembk}
For $0<R\leq \infty$. Let $A(x)=\sum_{n=0}^\infty a_nx^n$ and 
$C(x)=\sum_{n=0}^\infty c_nx^n$ be two real power series converging on the interval $(-R,R)$. If the sequence
$\{a_n/c_n\}$ is increasing (decreasing) and $c_n>0$ for all $n$, then the function $A(x)/C(x)$ is also
increasing (decreasing) on $(0,R)$.
\end{lemma}

For $|x|<\pi$, the following power series expansions can be found in \cite[1.3.1.4 (2)--(3)]{jef},
\begin{equation}\label{xcot}
x \cot x=1-\sum_{n=1}^\infty\frac{2^{2n}}{(2n)!}|B_{2n}|x^{2n},
\end{equation}

\begin{equation}\label{cot}
\cot x=\frac{1}{x}-\sum_{n=1}^\infty\frac{2^{2n}}{(2n)!}|B_{2n}|x^{2n-1},
\end{equation}
and 
\begin{equation}\label{coth}
{\rm \coth x}=\frac{1}{x}+\sum_{n=1}^\infty\frac{2^{2n}}{(2n)!}|B_{2n}|x^{2n-1},
\end{equation}
where $B_{2n}$ are the even-indexed Bernoulli numbers 
(see \cite[p. 231]{IR}). 
We can get the following expansions directly from (\ref{cot}) and (\ref{coth}),

\begin{equation}\label{cosec}
\frac{1}{(\sin x)^2}=-(\cot x)'=\frac{1}{x^2}+\sum_{n=1}^\infty\frac{2^{2n}}{(2n)!}
|B_{2n}|(2n-1)x^{2n-2},
\end{equation}

\begin{equation}\label{cosech}
\frac{1}{(\sinh x)^2}=-({\rm coth} x)'=\frac{1}{x^2}-\sum_{n=1}^\infty\frac{2^{2n}}{(2n)!}(2n-1)|B_{2n}|x^{2n-2}.
\end{equation}
For the following expansion formula 
\begin{equation}\label{xsin}
\frac{x}{\sin x}=1+\sum_{n=1}^\infty\frac{2^{2n}-2}{(2n)!}|B_{2n}|x^{2n}
\end{equation}
see \cite{li}.


For easy reference we recall the following lemma from \cite{barsan,
barsan2}.
\begin{lemma}\label{lemma1} For $a>b>0$, $x\in(0,\pi/2)$ and $y>0$, one has
$$
\frac{P}{A}= \frac{\sin (x)}{x},\ \frac{G}{A} = \cos(x),\, \frac{H}{A}= \cos(x)^2,\   
\frac{X}{A}= e^{x {\rm cot}(x)-1},  
$$

$$
\frac{L}{G}= \frac{\sinh (y)}{y},\, \frac{L}{A}= \frac{\tanh (y)}{y},\
 \frac{H}{G}= \frac{1}{\cosh (y)},\,  
\frac{Y}{G}= e^{\tanh (y)/y -1}. 
$$

$$
\log\left(\frac{I}{G}\right)=\frac{A}{L}-1,\quad 
\log\left(\frac{Y}{G}\right)=\frac{L}{A}-1.
$$
\end{lemma}


\begin{remark}\rm  It is well known that many inequalities involving the means can be obtain from the classical inequalities of trigonometric functions. For example, the following inequality
$$
e^{(x/\tanh(x)-1)/2}<\frac{\sinh(x)}{x},\quad x>0,
$$
recently appeared in \cite[Theorem 1.6]{barsan3}, which is equivalent to 
\begin{equation}\label{ineq1003}
\frac{\sinh(x)}{x}>e^{x/\tanh(x)-1}\frac{x}{\sinh(x)}.
\end{equation} 
By Lemma \ref{lemma1}, this can be written as
$$\frac{L}{G}>\frac{I}{G}\cdot \frac{G}{L}=\frac{I}{L},$$
or 
\begin{equation}\label{ineq6}
L>\sqrt{IG}.
\end{equation}
The inequality \eqref{ineq6} was proved by Alzer \cite{alzer2}. 


The following trigonometric inequalities 
(see \cite[Theorem 1.5]{barsan3}) imply an other double inequality for Seiffert mean $P$,

\begin{equation}\label{0209f}
\begin{cases}
\displaystyle\exp\left(\frac{1}{2}\left(\frac{x}{\tan x}-1\right)\right)
< \displaystyle\frac{\sin x}{x}\\
< \displaystyle
\exp\left(\left(\log\frac{\pi}{2}\right)\left(\frac{x}{\tan x}-1\right)\right)& x\in(0,\pi/2),\\
\sqrt{AX}<P<A\left(\frac{X}{A}\right)^{\log(\pi/2)}.
\end{cases}
\end{equation}
The second mean inequality in \eqref{0209f} was also pointed out by S\'andor 
(see \cite[Theorem 2.12]{sandornew}).
\end{remark}


\begin{lemma}\cite[Theorem 2]{avv1}\label{lem0}
For $-\infty<a<b<\infty$,
let $f,g:[a,b]\to \mathbb{R}$
be continuous on $[a,b]$, and differentiable on
$(a,b)$. Let $g^{'}(x)\neq 0$
on $(a,b)$. If $f^{'}(x)/g^{'}(x)$ is increasing
(decreasing) on $(a,b)$, then so are
$$\frac{f(x)-f(a)}{g(x)-g(a)}\quad and \quad \frac{f(x)-f(b)}{g(x)-g(b)}.$$
If $f^{'}(x)/g^{'}(x)$ is strictly monotone,
then the monotonicity in the conclusion
is also strict.
\end{lemma}


\begin{lemma}\label{lemma5} The following function
$$h(x)=\frac{\log(x/\sin(x))}{\log(e^{1-x/\tan(x)}\sin(x)/x)}$$
is strictly decreasing from $(0,\pi/2)$ onto $(\beta_2,1)$, where
$\beta_2=\log(\pi/2)/\log(2e/\pi)\approx 0.8234.$ In particular, 
for $x\in(0,\pi/2)$ we have
\begin{equation}\label{2102a}
\left(\frac{e^{1-x/\tan(x)}\sin(x)}{x}\right)^{\beta_2}<\frac{x}{\sin(x)}<\left(\frac{e^{1-x/\tan(x)}\sin(x)}{x}\right).
\end{equation}
\end{lemma}

\begin{proof} Let 
$$h(x)=\frac{h_1(x)}{h_2(x)}=\frac{\log(x/\sin(x))}{\log(e^{1-x/\tan(x)\sin(x)/x})},$$
for $x\in(0,\pi/2)$. Differentiating with respect to $x$, we get
$$\frac{h'_1(x)}{h'_2(x)}=\frac{1-x/\tan(x)}{(x/\sin(x))^2-1}=
\frac{A_1(x)}{B_1(x)}.$$
Using the expansion formula 
we have
$$A_1(x)=\sum_{n=1}^\infty\frac{2^{2n}2n}{(2n)!}|B_{2n}|x^{2n}=\sum_{n=1}^\infty a_nx^{2n}$$
and 
$$B_1(x)=\sum_{n=1}^\infty\frac{2^{2n}2n}{(2n)!}|B_{2n}|(2n-1)x^{2n}=\sum_{n=1}^\infty b_nx^{2n}.$$
Let $c_n=a_n/b_n=1/(2n-1)$, which is the decreasing in $n\in\mathbb{N}$. Thus, by Lemma \ref{lembk} $h'_1(x)/h'_2(x)$
is strictly decreasing in $x\in(0,\pi/2)$. In turn, this implies by Lemma 
\ref{lem0} that $h(x)$ is strictly decreasing in $x\in(0,\pi/2)$. Applying l'H\^opital rule, we get 
$\lim_{x\to 0}h(x)=1$ 
and $\lim_{x\to \pi/2}h(x)=\beta_2$. This completes the proof.
\end{proof}

\begin{remark}\rm It is observed that the inequalities in \eqref{2102a} coincide with the trigonometric inequalities given in \eqref{0209f}. Here Lemma \ref{lemma5} gives a new and an optimal proof for these inequalities.
\end{remark}


\begin{lemma}\label{lemma2} The following function
$$f(x)=\frac{1-e^{x/\tan(x)-1}}{1-\cos(x)}$$
is strictly decreasing from $(0,\pi/2)$ onto 
$((e-1)/e,2/3)$ where $(e-1)/e\approx 0.6321$.
In particular, for $x\in(0,\pi/2)$, we have
$$\frac{1}{\log(1+(e-1)\cos(x))}<\frac{\tan(x)}{x}<
\frac{1}{1+\log((1+2\cos(x))/3)}.$$
\end{lemma}

\begin{proof} Write $f(x)=f_1(x)/f_2(x)$, where 
$f_1(x)=1-e^{x/\tan(x)-1}$ and $f_2(x)=1-\cos(x)$ for all $x\in(0\pi/2)$. Clearly,
$f_1(x)=0=f_2(x)$. Differentiating with respect to $x$, we get
$$\frac{f'_1(x)}{f'_2(x)}=\frac{e^{x/\tan(x)-1}}{\sin(x)^3}\left(\frac{x}{\sin(x)^2}-\frac{\cos(x)}{\sin(x)}\right)
=f_3(x).$$ Again
$$f'_3(x)=-\frac{e^{x/\tan(x)-1}}{\sin(x)^3}\left(c(x)-2\right),$$ 
where
$$c(x)=x
\left(\frac{\cos(x)}{\sin(x)}+\frac{x}{\sin(x)^2}\right).$$
In order to show that $f'_3<0$, it is enough to prove that 
$$c(x)>2,$$
which is equivalent to 
$$\frac{\sin(x)}{x}<\frac{x+\sin(x)\cos(x)}{2\sin(x)}.$$ Applying the Cusa-Huygens inequality
$$\frac{\sin(x)}{x}<\frac{\cos(x)+2}{3},$$ we get
$$\frac{\cos(x)+2}{3}<\frac{x+\sin(x)\cos(x)}{2\sin(x)},$$
which is equivalent to $(\cos(x)-1)^2>0$. Thus $f'_3 >0$, clearly $f'_1/f'_2$ is strictly decreasing in $x\in(0,\pi/2)$. By Lemma \ref{lem0}, we conclude that the function $f(x)$ is strictly decreasing in 
$x\in(0,\pi/2)$. The limiting values follows easily. This completes the proof of the lemma.  
\end{proof}


\begin{lemma}\label{lemma3} The following function
$$f_4(x)=
\frac{\sin (x)}{x \left(\cos (x)-e^{x \cot (x)-1}+1\right)}$$
is strictly increasing from $(0,\pi/2)$ onto $(1,c)$, where
$c=2e/(\pi(e-1))\approx 1.0071$. In particular, for $x\in(0,\pi/2)$ we have
$$1+\cos(x)-e^{x/\tan(x)-1}<\frac{\sin(x)}{x}<c(1+\cos(x)-e^{x/\tan(x)-1}).$$
\end{lemma}

\begin{proof} Differentiating with respect to $x$ we get
$$f'_4(x)=\frac{e (x-\sin (x)) \left(e \cos (x)-(x+\sin (x)) e^{x \cot (x)} \csc
   (x)+e\right)}{x^2 \left(e \cos (x)-e^{x \cot (x)}+e\right)^2}.$$
	Let
	$$f_5(x)=\log \left((x+\sin (x)) e^{x \cot (x)} /\sin (x)\right)-\log (e \cos (x)+e),$$
	for $x\in(0,\pi/2)$. Differentiation yields 
	$$f'_5(x)=\frac{2-x \left(\cot (x)+x \csc ^2(x)\right)}{x+\sin (x)},$$
	which is negative by the proof of Lemma \ref{lemma2}, and $\lim_{x\to 0}f_5(x)=0$. This implies that $f'_4(x)>0$, and $f_4(x)$
	is strictly increasing. The limiting values follows easily. This implies the proof.
\end{proof}


\begin{lemma} For $a\neq b$, one has
\begin{equation}\label{2602a0}
M_{1/3}<(2G+A)/3.
\end{equation}
\end{lemma}

\begin{proof} Let $G=G(a,b)$, etc. Divide both sides with $b$ and put $a/b=x$. Then inequality 
\eqref{2602a0} becomes the following:
\begin{equation}\label{2602a}
\left(\frac{x^{1/3}+1}{2}\right)^3 < 4(x+4\sqrt{x}+1). 
\end{equation}   
Let $x=t^6$, where  $t>1$. Then raising both sides of \eqref{2602a} to $3$th power, after elementary transformations we get, 
   $$t^6-9t^4+16t^3-9t^2+1>0,$$
which can be written as $(t-1)^4(t^2+4t+1)>0$, so it is true. Thus \eqref{2602a} and \eqref{2602a0} are proved.
\end{proof}

Since $L< M_{1/3}$, by \eqref{2602a0} we get a new proof , as well as a refinement of Carlson's inequality  $L< (2G+A)/3$.


\begin{lemma} For $a\neq b$, one has
\begin{equation}\label{2602b0}
H_{1/2}<(2G+A)/3.
\end{equation}
\end{lemma}

\begin{proof}
By definition of $H_\alpha$ one has  
$$H_{1/2}=((\sqrt{a}+ (ab)^{1/4}+\sqrt{b})/3)^2=(\sqrt{2(A+G)} +\sqrt{G})^2/9,$$
by remarking that $\sqrt{a}+\sqrt{b}=\sqrt{2(A+G)}$.
Therefore, (2) can be written equivalently as 
\begin{equation}\label{2602-7}
(2(A+G) +2\sqrt{2G(A+G)} +G)/9< (2G+A)/3.
\end{equation}
Now, it is immediate that \eqref{2602-7} becomes, after elementary computations
\begin{equation}\label{2602-8}
A+3G>2\sqrt{2G(A+G)},
\end{equation}
or by raising both sides to the $2$th power:
$$A^2+6AG+9G^2> 8AG+8G^2,$$
which become  $(A-G)^2>0$, true. Thus \eqref{2602-8} and \eqref{2602-7} are proved, and \eqref{2602b0} follows.
\end{proof}
\section{\bf Proof of main result}

\noindent{\bf Proof of Theorem \ref{thm1}.} It follows from Lemma \ref{lemma2} that
$$\frac{e-1}{e}<\frac{1-1/e^{1-x/\tan(x)}}{\cos(x)/e^{1-x/\tan(x)}-1/e^{1-x/\tan(x)}}<\frac{2}{3}.$$
Now we get the proof of \eqref{30ineqa} by utilizing the Lemma \ref{lemma1}.
The proof of \eqref{30ineqb}
follows easily from Lemmas \ref{lemma1} and \ref{lemma2}.
$\hfill\square$

\bigskip

%

\noindent{\bf Proof of Theorem \ref{2602thm0}.} The second inequality of \eqref{2402a} is right side of relation \eqref{89ineq}.
In \cite{alzer1}, Alzer and Qiu proved the third inequality of \eqref{2402a}.
The last inequality is the left side of \eqref{2602alzer}.
By \cite{chuet} and \cite{alzer1}, $q$ is best possible constant in both sides.

Now we shall prove the first inequality of \eqref{2402a}.
By using Lemma \ref{lemma1}, is is easy to see that, this becomes equivalent with   
$1+\cos (x)< e^{x\cot (x)}$, 
or
\begin{equation}\label{2602tri}
\log(1+\cos (x))< x\cot (x),\quad   x \in (0, \pi/2).
\end{equation}
Now, by the classical inequality  $\log(1+t)< t\;  (t>0)$, applied to $t=\cos (x)$, we get
$\log(1+\cos (x))< \cos (x)$. 
Now $\cos (x)< x\cot (x)  = x\cos (x)/\sin (x)$ is true by  $\sin (x) <x$.
The proof of \eqref{2602tri} follows.
$\hfill\square$

One has the following relation, in analogy with relation \eqref{2602alzer} of 
Theorem \ref{sandorbhayo-theorem} for the mean $Y$:

\begin{corollary} One has
$$(A+G)/e<X< (A+G)/2,$$ 
where the constants  e and 2 are best possible.
\end{corollary}
The inequalities  $(A+G)/e < X$ and 
$(2G+A)/3 <X$ are not comparable.

\bigskip

\noindent{\bf Proof of Theorem \ref{2602thm1}.} 
The second inequality of \eqref{2402ee} appeared in \cite{sandortwosharp}  in the form $P^2>AX$. The last inequality follows by $P< (2A+G)/3$. Indeed, one has  $((2A+G)/3)^2< A(A+G)/2$ becomes 
$2G^2< A^2+AG$, and this is true by $G<A$.
$\hfill\square$

\bigskip

\noindent{\bf Proof of Theorem \ref{2602thm2}.}
By \cite[Theorem 2.10]{sandornew}, one has $P+X> A+G$, and remarking that $(A+G)/2 =M_{1/2}$, the left side of \eqref{2402d} follows.
For the right side of \eqref{2402d}, we will use  $P< M_{t}$ with $t= 2/3$ 
(see  \cite{sandor1405}), and $X< M_q$ (\cite{chuet}), where $q=(\log 2)/(\log 2+1)$.
On the other hand the function $f(t) = M_t$  is known to be strictly $\log$-concave for $t>0$ 
(see \cite{sandorlog}). Particularly, this implies that $f(t)$ is strictly concave. Thus 
$(M_t+M_q)/2 < M_{(t+q)/2}$.     
As  $(t+q)/2 = k \approx 0.5380$, the result follows.
$\hfill\square$

\begin{corollary}\label{cor27} One has the followng two sets of inequalities:
\begin{enumerate}
\item $PX>PL> AG$,     

\item $IL >PL> AG$.  
\end{enumerate}
\end{corollary}

\begin{proof} The first inequality of (1) follows by  $X>L$, while the second appears in 
\cite{sandor1405}.
The first inequality of (2) follows by $I>P$, while the second one is the same as the second one in (1).
\end{proof}

\begin{remark}\rm
Particularly in Corollary \ref{cor27}, (2) improves Alzer's inequality $IL>AG$.
Inequality (1) improves  $PX>AG$, which appears in \cite{sandornew}.
\end{remark}

\begin{corollary} One has
\begin{enumerate}
\item $X > A(P+G)/(3P-G)> (2G+A)/3> L.$
\item $P^2/A >X> (P+G)/2.$  
\end{enumerate}
\end{corollary}

\begin{proof} The first two inequalities  of (1) appear in 
\cite[Theorem 2.5 and Remark 2.3]{sandornew}. The second  inequality of  (2) follows by the first inequality of (1) and the remark that $A/(3P-G)>1/2$ , since this is $P< (2A+G)/3$; while the first one is $P^2>AX$ (\cite{sandortwosharp}).
\end{proof}

\begin{remark}\rm
Since it is known that $P> (2/\pi)A$
(due to Seiffert, see \cite{sandor1405}).
By $X> (P+G)/2$ we get the inequality
$X> [(2/\pi)A+G]/2$, 
which is not comparable with $(A+G)/e<X$.
\end{remark}

\noindent{\bf Proof of Theorem \ref{2602-thm}.} The first inequality of \eqref{2602-9} follows, since the function $f(t)=M_t$ is known to be strictly increasing. The second inequality follows by \eqref{2602a0}, while the third one can be found in Theorem \ref{sandorbhayo-theorem}.

It is known  that $H_p$ is an increasing function of $p$.
Therefore, the proof of \eqref{2602-10} follows by \eqref{2602b0}. 
$\hfill\square$

%

\begin{corollary} For $a,b>0$ with $a\neq b$, we have
\begin{equation}\label{2202a}
\frac{I}{L}<\frac{L}{G}<1+\frac{G}{H}-\frac{I}{G}.
\end{equation}
\end{corollary}

\begin{proof} The first inequality is due to Alzer \cite{alzer2}, while the second inequality follows from the fact that the function
$$x\mapsto \frac{1-e^{x/\tanh(x)-1}}{1-\cosh(x)}:(0,\infty)\to (0,1)$$
is strictly decreasing. The proof of the monotonicity of the function is the analogue to the proof of Lemma \ref{lemma2}.
\end{proof}

The right side of \eqref{2202a} may be written as 
$L+I< G+A$ (by $H=G^2/A$), and this is due to Alzer (see \cite{alzer1, sandorc} for history of early results).

\bigskip

\noindent{\bf Proof of Theorem \ref{thm30}.} The proof follows easily from Lemma \ref{lemma5}.
$\hfill\square$

\bigskip

In \cite{Seif2}, Seiffert proved that
\begin{equation}\label{seifineq}
\frac{2}{\pi}A<P,
\end{equation}
for all $a,b>0$ with $a\neq 0$.
As a counterpart of the above result we give the following inequalities.

\begin{corollary}\label{coro89} For $a,b>0$ with $a\neq b$, the following inequalities 
$$\frac{1}{e}A<\frac{\pi}{2e} P <X<P$$
holds true.
\end{corollary}

\begin{proof} The first inequality follows from \eqref{seifineq}. For the proof of the second inequality we write by Lemma \ref{lemma1}
$$f'_5(x)=\frac{X}{P}=\frac{xe^{x/\tan(x)-1}}{\sin(x)}=f_5(x)$$
for $x\in(0,\pi/2)$. Differentiation gives
$$\frac{e^{x/tan(x)-1}}{\sin(x)}\left(1-\frac{x^2}{\sin(x)^2}\right)<0.$$ Hence the function $f_5$ is strictly decreasing in $x$, with 
$$\lim_{x\to 0}f_5(x)=1 \quad {\rm and}\quad \lim_{x\to \pi/2}f_5(x)=\pi/(2e)\approx 
0.5779.$$ This implies the proof.
\end{proof}

We finish this paper by giving the following open problem and a conjecture.

\bigskip

\noindent {\bf Open problem.} {\it 
What are the best positive  constants $a$ and $b$, such that}
$$M_a< (P+X)/2< M_b.$$

\bigskip

\noindent {\bf Conjecture.} {\it For $a\neq b$, one has}
$$PX > IL.$$

%


\end{document}